\documentclass[11pt]{article}
\usepackage{amssymb}
\usepackage{amsmath}
\usepackage{amsthm}
\usepackage{latexsym}
\usepackage[english]{babel}
\usepackage[overload]{empheq}
\usepackage{cleveref}

\usepackage{amsmath,amssymb,amscd,enumerate,eucal,bm}
\usepackage{graphicx}
\usepackage{amsmath}
\usepackage{amsfonts}
\usepackage{amssymb}
\usepackage{amsmath}
\usepackage{amsthm}
\usepackage{latexsym}
\usepackage{graphicx}
\usepackage{cases}
\usepackage{float}
\usepackage{caption}
\usepackage{subfigure}
\usepackage{fullpage}
\usepackage{tikz}
\usetikzlibrary{matrix,arrows,decorations.pathmorphing}

\numberwithin{equation}{section}

\newcommand{\R}{\mathbb{R}}

\newcommand{\e}{\varepsilon}

\newcommand{\Sph}{\mathbb{S}^3}

\newcommand{\la}{\langle}
\newcommand{\ra}{\rangle}
\newcommand{\vol}{\rm vol}

\newcommand{\dvol}{d{\rm vol}}

\newtheorem{theorem}{Theorem}[section]

\newtheorem{lemma}[theorem]{Lemma}

\newtheorem{proposition}[theorem]{Proposition}
\newtheorem{remark}[theorem]{Remark}

\begin{document}

\title{Existence results for the conformal Dirac-Einstein system}

\author{Chiara Guidi$^{(1)}$ \& Ali Maalaoui$^{(2)}$ \& Vittorio Martino$^{(3)}$}
\addtocounter{footnote}{1}
\footnotetext{Dipartimento di Matematica, Universit\`a di Bologna, piazza di Porta S.Donato 5, 40126 Bologna, Italy. E-mail address:
{\tt{chiara.guidi12@unibo.it}}}
\addtocounter{footnote}{1}
\footnotetext{Department of mathematics and natural sciences, American University of Ras Al Khaimah, PO Box 10021, Ras Al Khaimah, UAE. E-mail address:
{\tt{ali.maalaoui@aurak.ae}}}
\addtocounter{footnote}{1}
\footnotetext{Dipartimento di Matematica, Universit\`a di Bologna, piazza di Porta S.Donato 5, 40126 Bologna, Italy. E-mail address:
{\tt{vittorio.martino3@unibo.it}}}

\date{}
\maketitle

\vspace{5mm}

{\noindent\bf Abstract} {\small In this paper we consider the coupled system given by the first variation of the conformal Dirac-Einstein functional. We will show existence of solutions by means of perturbation methods.}

\vspace{5mm}


\vspace{5mm}


\vspace{5mm}


\section{Introduction}

\noindent
Let $(M,g,\Sigma M)$ be a closed (compact, without boundary) three dimensional Riemannian Spin manifold where $\Sigma M$ is its spin bundle. We denote by $L_g$ the conformal Laplacian of $g$ and by $D_g $ the Dirac operator. We consider the energy functional
\begin{equation}\label{eq: energy functional}
E_M(v,\psi)=\frac{1}{2}\left(\int_{M}v L_{g}v+\langle D_{g} \psi,\psi \rangle -|v|^{2}|\psi|^{2} \dvol_{g}\right)
\end{equation}
and we take its first variation on the related Sobolev space $H^1(M) \times H^{\frac{1}{2}}(\Sigma M)$; therefore its critical points satisfy the coupled system
\begin{equation} \label{el}
\begin{cases}
L_{g} v=|\psi|^{2}u\\
\\
D_{g}\psi=|v|^{2}\psi
\end{cases} \text{on } M.
\end{equation}

\noindent
This functional arises as the conformal version in the description of a super-symmetric model consisting of coupling gravity with fermionic interaction and it generalizes the classical Hilbert-Einstein energy functional, see for instance \cite{Belg, Fin, Kim}.\\
Indeed, the total energy functional consists of the Hilbert-Einstein energy which is the total curvature, coupled with a fermionic action. Now, since the energy of the system is invariant under the group of diffeomorphisms of $M$, when one restricts it to a fixed conformal class of a given Riemannian metric $g$, the functional $E_M$ shows up.\\
In particular, due to the conformal invariance, the Palais-Smale compactness condition is violated by this functional and in addition, due to the presence of the Dirac operator, it is strongly indefinite.\\
Regarding the first issue, in \cite{MMdiraceinstein} the authors studied the lack of compactness and gave a precise description of the bubbling phenomena, characterizing the behaviour of the Palais-Smale sequences, in the spirit of classical works \cite{stru84, sauh81, wen80, lio1-85, lio2-85, brecor85, bahcor88}. For the strongly indefinite difficulty, in \cite{M,MV,MV1} general functionals with these features are studied by using methods based on a homological approach. Notice that so far, one cannot apply these homological approaches because of the violation of compactness stated above.

\noindent
In this paper, we are concerned with the existence of solutions to the coupled system, by using a perturbation approach, starting from the sphere $\Sph$ equipped with its standard metric $g_{\Sph}$.\\ 
Therefore, let $K$ be a function of the form $K=1+\e k$, where $k$ is a function with suitable assumptions to be determined later; we consider the functional
\begin{equation}
\mathcal{E}(v,\psi)=\frac{1}{2}\left(\int_{\Sph}v L_{g_{\Sph}}v+\langle D_{g_{\Sph}} \psi,\psi \rangle -K|v|^{2}|\psi|^{2} d\vol_{g_{\Sph}}\right)
\end{equation}
and we will focus on the existence of solutions to the following coupled system:
\begin{equation}\label{eq: pb on S}
\begin{cases}
L_{g_{\Sph}}v=K|\psi|^2v\\
\\
D_{g_{\Sph}}\psi=Kv^2\psi\quad
\end{cases}\text{on }\Sph
\end{equation}
Notice that these solutions converge to the standard bubbles when the parameter $\varepsilon$ tends to zero. This is expected from the description of the Palais-Smale sequences of the functional $E_{M}$, but it remains open whether all the solutions on the sphere with positive scalar component are in fact standard ones. 

\noindent
Let us denote by $\pi:\Sph\setminus\{sp\}\to\R^3$ the stereographic projection, where $sp$ is the south pole. Our main result is the following
\begin{theorem}\label{thm: main}
Let $k\in C^2(\Sph)$ be a Morse function on $\Sph$ such that the south pole is not a critical point. Let us set $h=k\circ \pi^{-1}$ and suppose that
$$(i) \quad \Delta h(\xi)\neq  0,\;\forall\; \xi\in\textnormal{crit}[h]\; ,$$
$$(ii) \qquad \sum_{\substack{\xi\in\textnormal{crit}[h]\\ \Delta h(\xi)<0}}(-1)^{m(h,\xi)}\neq -1,$$
where $\Delta$ is the standard Laplacian operator on $\R^3$, $\textnormal{crit}[h]$ denotes the set of critical points of $h$ and $m(h,\xi)$ is the morse index of $h$ at a critical point $\xi$.\\
Then, there exists $\e_0>0$ such that for $K=1+\e k$ and $|\e|<\e_0$, the system \eqref{eq: pb on S} has a solution.
\end{theorem}

\noindent
The condition on the critical point at the south pole of the sphere is needed since we are going to use the standard stereographic projection $\pi$, however this condition can be always satisfied by making a unitary transformation which does not affect the generality of the result.

\noindent
The previous result is the analogous of several ones obtained with this kind of hypothesis of Bahri-Coron type on the function $k$: for instance, for the standard Riemannian case of prescribing the scalar curvature and its generalization to the $Q_\gamma$ curvature see \cite{ambgarper, changyang, chenzheng2014}; in the case of prescribing the Webster curvature in the CR setting and its fractional generalization see \cite{malchiodi-uguzzoni 2002} and \cite{chenwang2017}; for the spinorial Yamabe type equations involving the Dirac operator on the sphere see \cite{I}.

\noindent
The idea of the proof follows the abstract perturbation method introduced in \cite{AB}.\\
The difficulties in our situation come from the fact of having a system, from the strongly indefiniteness of one of the operator involved and finally from the degeneracy of the critical points of the finite dimensional reduction of the functional, which is due to the invariance with respect to one of the parameters of the problem (see Remark \ref{rmk: degeneracy}).


\section{Notations and definitions}

\noindent
Let $(M,g)$ be a closed (compact, without boundary) three dimensional Riemannian manifold.\\
We start to describe shortly the first operator appearing in the system. We denote by $L_g$ the conformal Laplacian acting on functions
$$L_g =-\Delta_g +\frac{1}{8}R_g.$$
Here $\Delta_g$ is the standard Laplace-Beltrami operator and $R_g$ is the scalar curvature. $L_g$ is a conformally invariant operator. More precisely, given a metric $\tilde{g}=f^2g$ in the conformal class of $g$, we have
$$L_{\tilde g}u=f^{-\frac{5}{2}}L_g(f^{\frac{1}{2}}u).$$
We recall that the usual Sobolev space on $M$, denoted by $H^1(M)$, continuously embeds in $L^p(M)$ for $1\leq p \leq 6$. Moreover, for $1\leq p<6$, the embedding is compact.\\
In particular, if we assume $M$ to be the sphere
$$\Sph=\{(x',x_4)\in\R^3\times\R:\; |x'|^2+x_4^2=1\}$$
equipped with its standard metric $g_{\Sph}$, then it is possible to identify $\Sph\setminus\{sp\}$, being $sp=(0,-1)$ the south pole, with $\R^3$, by means of the stereographic projection
\begin{align*}
\pi:\Sph\setminus\{sp\}&\to\R^3\\
(x,x_4)&\mapsto y=\frac{x}{1+x_4}.
\end{align*}
The standard metric $g_{\R^3}$ on $\R^3$ and the metric $\tilde{g}=(\pi^{-1})^*g_{\Sph}$ are conformal, more precisely $\tilde{g}=f^2 g_{\R^3}$, with $f=\frac{2}{1+|y|^2}$. Thus the standard conformal Laplacian on the sphere $L_{g_{\Sph}}$ and the one on $\R^3$, which we denote as usual $L_{g_{\R^3}}=-\Delta$, are related by the following identity
\begin{equation}\label{eq: relation LS LE}
L_{g_{\Sph}}v=\left[f^{-\frac{5}{2}}(-\Delta)\left(f^{\frac{1}{2}}v\circ\pi^{-1} \right)\right]\circ\pi, \quad v\in H^{1}(\Sph).
\end{equation}

\noindent
Now, let us describe the second operator involved. Let $\Sigma M$ be the canonical spinor bundle associated to $M$, whose sections are simply called spinors on $M$. This bundle is endowed with a natural Clifford multiplication
$$\text{Cliff}:C^\infty(TM\otimes\Sigma M)\longrightarrow C^\infty(\Sigma M), $$
a hermitian metric and a natural metric connection
$$\nabla^\Sigma:C^\infty(\Sigma M)\longrightarrow C^\infty(T^*M\otimes\Sigma M).$$
We denote by $D_g$ the Dirac operator acting on spinors
\begin{align*}
D_g:C^\infty &(\Sigma M)\longrightarrow C^\infty(\Sigma M)\\
D_g&=\text{Cliff} \circ \nabla^\Sigma
\end{align*}
where the composition $\text{Cliff} \circ \nabla^\Sigma$ is meaningful provided that we identify $T^*M \simeq TM$ by means of the metric $g$. We also have a conformal invariance that in our situation, $\tilde g= f^2g$, reads as follows: there exists an isomorphism of vector bundles $F:\Sigma(M,g)\to\Sigma(M,\tilde{g})$ such that
\begin{equation}\label{eq: relation Dg Dtildeg}
D_{\tilde g}\psi=F\left[f^{-2}D_g\left(fF^{-1}\psi\right)\right].
\end{equation}
The functional space that we are going to define is the Sobolev space $H^{\frac{1}{2}}(\Sigma M)$. First we recall that the Dirac operator $D_g$ on a compact manifold is essentially self-adjoint in $L^2(\Sigma M)$, has compact resolvent and there exists a complete $L^2$-orthonormal basis of eigenspinors $\{\psi_i\}_{i\in\mathbb{Z}}$ of the operator
$$D_g\psi_i=\lambda_i \psi_i ,$$
and the eigenvalues $\{\lambda_i\}_{i\in\mathbb{Z}}$ are unbounded, that is $|\lambda_i|\rightarrow\infty$, as $|i|\rightarrow\infty$.
In this way every function in $L^{2}(\Sigma M)$, it has a representation in this basis, namely:
$$\displaystyle \psi=\sum_{i\in \mathbb{Z}}a_{i}\psi_{i}, \qquad \psi\in L^{2}(\Sigma M).$$
We define the unbounded operator $|D_g|^{s}: L^{2}(\Sigma M)\rightarrow L^{2}(\Sigma M)$ by
$$|D_g|^{s}(\psi)=\sum_{i\in \mathbb{Z}} a_{i}|\lambda_{i}|^{s}\psi_{i}$$
and we denote by $H^s(\Sigma M)$ the domain of $|D_g|^{s}$, namely $\psi\in H^s(\Sigma M)$ if and only if
$$\sum_{i\in \mathbb{Z}} a_{i}^2|\lambda_{i}|^{2s}<+\infty .$$
$H^s(\Sigma M)$ coincides with the usual Sobolev space $W^{s,2}(\Sigma M)$ and for $s <0$, $H^s(\Sigma M)$ is defined as the dual of $H^{-s}(\Sigma M)$.\\
For $s >0$, we define the inner product, for $\psi,\phi \in H^s(\Sigma M)$
$$\langle \psi,\phi\rangle_{s}=\langle|D_g|^{s}\psi,|D_g|^{s}\phi\rangle_{L^{2}},$$
which induces an equivalent norm in $H^{s}(\Sigma M)$; we will take
$$\langle \psi,\psi\rangle:=\langle \psi,\psi\rangle_{\frac{1}{2}}=\|\psi\|^{2}$$
as our standard norm for the space $H^{\frac{1}{2}}(\Sigma M)$. In this case as well, the embedding $H^{s}(\Sigma M) \hookrightarrow L^p(\Sigma M)$ is continuous for $1\leq p \leq 3$ and it is compact if $1\leq p <3$.\\
Then, we decompose $H^{\frac{1}{2}}(\Sigma M)$ in a natural way. Let us consider the $L^2$-orthonormal basis of eigenspinors $\{\psi_i\}_{i\in\mathbb{Z}}$: we denote by $\psi_i^-$ the eigenspinors with negative eigenvalue, $\psi_i^+$ the eigenspinors with positive eigenvalue and $\psi_i^0$ the eigenspinors with zero eigenvalue; we also recall that the kernel of $D_g$ is finite dimensional. Now we set:
$$H^{\frac{1}{2},-}:=\overline{\text{span}\{\psi_i^-\}_{i\in\mathbb{Z}}},\quad
H^{\frac{1}{2},0}:=\text{span}\{\psi_i^0\}_{i\in\mathbb{Z}}, \quad
H^{\frac{1}{2},+}:=\overline{\text{span}\{\psi_i^+\}_{i\in\mathbb{Z}}},$$
where the closure is taken with respect to the $H^{\frac{1}{2}}$-topology. Therefore we have the orthogonal decomposition of $H^{\frac{1}{2}}(\Sigma M)$, which reads as:
$$H^{\frac{1}{2}}(\Sigma M)=H^{\frac{1}{2},-}\oplus H^{\frac{1}{2},0}\oplus H^{\frac{1}{2},+}.$$
Also, we let $P^{+}$ and $P^{-}$ be the projectors on $H^{\frac{1}{2},+}$ and $H^{\frac{1}{2},-}$ respectively.\\
Again, if we assume $M$ to be the sphere $\Sph$ and we identify $\Sph$ minus the south pole with $\R^3$ via stereographic projection, the conformal invariance of the Dirac operator reads as
\begin{equation}\label{eq: relation DS DE}
D_{g_{\Sph}}\psi=F\left\{\left[f^{-2}D\left(fF^{-1}(\psi\circ \pi^{-1})\right)\right]\circ\pi\right\}, \quad\psi\in H^{\frac{1}{2}}(\Sigma\Sph)
\end{equation}
where $D_{g_{\Sph}}$ and $D_{g_{\R^3}}=D$ denote the Dirac operators on the standard sphere and $\R^3$ respectively; moreover $f=\frac{2}{1+|y|^2}$ and $F:\Sigma(\R^{3},g_{\R^3})\to \Sigma(\Sph,g_{\Sph})$ the isomorphism of vector bundles in \eqref{eq: relation Dg Dtildeg}.\\
In the sequel we will need the following function spaces on $\R^3$:
\begin{align*}
D^{\frac{1}{2}}(\Sigma\R^{3})&=\left\{\psi \in L^{3}(\Sigma\R^{3}):|\xi|^{\frac{1}{2}}|\widehat{\psi}|\in L^{2}(\R^{3})\right\}\; ;\\
D^{1}(\R^{3})&=\left\{u\in L^{6}(\R^{3}): |\nabla u| \in L^{2}(\R^{3})\right\}.
\end{align*}
Here $\widehat{\psi}$ is the Fourier transform of $\psi$.


\section{Proof of the main result}

\noindent
Our existence result will be obtained by means of the abstract perturbation method illustrated in \cite{AB}.\\
We recall it in the following theorem and then we will show how it can be applied in our setting.

\begin{theorem}\label{thm: abstract perturbation method}(see \cite{AB})
Let $A$ be an Hilbert space and assume $J_0\in C^2(A,\R)$ satisfies the following conditions
\begin{enumerate}
\item $J_0$ has a finite-dimensional manifold $Z$ of critical points,\label{assumption 1}
\item $J_0''(z)$ is a Fredholm operator of index zero for every $z\in Z$,\label{assumption 2}
\item $T_zZ=\textnormal{ker} J_0''(z)$, for every $z\in Z$.\label{assumption 3}
\end{enumerate}
For $G\in C^2(A,\R)$, we denote by $J_{\e}=J_0-\e G$ the perturbed functional, by $V$ the orthogonal complement of $T_zZ$ in A and by $P:A\to V$ the orthogonal projection. Then, for any $z\in Z$ there exists $v(z)\in V$ such that $P(J_{\e}'(z+v(z)))=0$. \\
\noindent
Moreover, if there exists a compact set $\Omega\subset Z$ such that $J_{\e}|_{Z}$ has a critical point $z\in\Omega,$ then $z+v(z)$ is a critical point of the perturbed functional $J_{\e}$ in $A$.
\end{theorem}

\noindent
In order to apply the previous result to our situation, we introduce the following map
$$H^{1}(\Sph)\times H^{\frac{1}{2}}(\Sigma\Sph) \ni (v,\psi)\mapsto(u,\phi)=\left(f^{\frac{1}{2}}v\circ\pi^{-1},fF^{-1}(\psi\circ\pi^{-1})\right) \; ,$$
which gives a one to one correspondence between solutions to \eqref{eq: pb on S} on $\Sph$ and solutions to the equivalent system on $\R^3$
\begin{equation}\label{eq: pb on R}
\begin{cases}
-\Delta u=H|\phi|^2u\\
\\
D\psi=Hu^2\phi\quad
\end{cases}\text{on }\R^3
\end{equation}
where we set $H=K\circ \pi^{-1}$. Hence let us consider this last problem and let us denote
$$A=D^{1}(\R^{3})\times D^{\frac{1}{2}}(\Sigma\R^{3}) \; .$$
We take $w=(u,\psi)\in A$ and we set
\begin{align*}
\quad J_0(w)=\frac{1}{2}\int_{\R^3} -u\Delta u+\la D\phi,\phi\ra-|u|^2|\phi|^2 \; ,\\
G(w)=\frac{1}{2}\int_{\R^3}h|u|^2|\phi|^2 \; ,\qquad J_{\varepsilon}(w)=J_0(w)-\varepsilon G(w)
\end{align*}
with $h=k\circ \pi^{-1}$. We are going to define the manifold of critical points of $J_0$. Let $\lambda\in\R^+$, $y,\xi\in\R^3$, $a\in\Sigma\R^3$ with $|a|=1$, it is well known that the functions
$$\bar U_{\lambda,\xi}(y)=\sqrt[4]{3}\frac{\lambda^{1/2}}{(\lambda^2+|y-\xi|^2)^{1/2}}$$
are a family of positive solutions to $-\Delta u=u^5$ in $\R^3$
and the spinors
$$\bar \Phi_{\lambda,\xi,a}(x)=\frac{2\lambda}{(\lambda^2+|y-\xi|^2)^{3/2}}\left(\lambda-(y-\xi)\right)\cdot a$$
solve $D\phi=\frac{3}{2}|\phi|\phi$ in $\Sigma\R^3$.
Using this fact, and the equality $|\bar{\Phi}_{\lambda,\xi,a} |=\frac{2}{1+|y|^2}$ , one can check that the pairs
$$\left( U_{\lambda,\xi},\Phi_{\lambda,\xi,a}\right)=\left(\sqrt[4]{3}\bar U_{\lambda,\xi},\frac{\sqrt{3}}{2}\bar \Phi_{\lambda,\xi,a}\right)\in A $$
are critical points of $J_0$. Hence
\begin{equation*}
Z=\left\{W_{\lambda,\xi,a}=(U_{\lambda,\xi},\Phi_{\lambda,\xi,a})\;:\lambda\in\R^+,\;\xi\in\R^3\;\text{and } a\in\Sigma\R^3, |a|=1\right\}\subset A
\end{equation*}
is a 7-dimensional manifold of critical points of $J_0$. Let us fix any $a_0\in\Sigma\R^3$ with $|a_0|=1$, in the sequel we will use the notation $U_0=U_{1,0}$, $\Phi_0=\Phi_{1,0,a_0}$ and $W_0=(U_0,\Phi_0)$.\\
Now we will check assumption 2 in Theorem \ref{thm: abstract perturbation method}. We have
\begin{multline*}
\la J_0''(W_{\lambda,\xi,a})[w_1],w_2\ra=\int_{\R^3}-u_2\Delta u_1-u_2u_1|\Phi_{\lambda,\xi,a}|^2-2u_2U_{\lambda,\xi}\la \Phi_{\lambda,\xi,a},\phi_1\ra \\
+\int_{\mathbb{R}^3}\la D\phi_1-|U_{\lambda,\xi}|^2\phi_1,\phi_2\ra-2U_{\lambda,\xi}u_1\la\Phi_{\lambda,\xi,a},\phi_2\ra \; .
\end{multline*}
Therefore $J''_0$ is a compact perturbation of the identity, hence it is a Fredholm operator of index zero for all $W_{\lambda,\xi,a}\in Z$.\\
Now it remains to check that $T_{W_{\lambda,\xi,a}}Z=\text{ker}J''_0(W_{\lambda ,\xi, a})$ for every $\lambda\in\R^+$, $\xi\in\R^3$ and $a\in\Sigma\R^3$ with $|a|=1$. Since $J_0''$ is invariant with respect to translations and dilations it will be enough to prove $T_{W_0}Z=\text{ker}J''_0(W_0).$ We will need the following Remark.

\begin{remark}\label{rmk: rescaling}
Let $\lambda_1=\frac{3}{4}$ and $\mu_1=\frac{3}{2}$. The map $(v,\psi)\mapsto (\nu,\eta)=(\mu_1^{-\frac{1}{2}}v,\lambda_1^{-\frac{1}{2}}\psi)$ is a one to one correspondence between solution to \eqref{eq: pb on S} on $\Sph$ and the equivalent rescaled system
\begin{equation}
\begin{cases}\label{eq: pb on S rescaled}
L_{g_{\Sph}}\nu=\lambda_1|\eta|^2\nu\\
\\
D_{g_{\Sph}}\eta=\mu_1\nu^2\eta
\end{cases}\quad \text{on }\Sph
\end{equation}
which in turn it is equivalent to
\begin{equation}
\begin{cases}\label{eq: pb on R rescaled}
-\Delta u=\lambda_1|\phi|^2u\\
\\
D\phi=\mu_1u^2\phi
\end{cases}\quad \text{on }\R^{3}
\end{equation}
by means of the stereographic projection. Notice that \eqref{eq: pb on R rescaled} arises as the first variation of the functional
$$\tilde{J}_0(w)=\frac{1}{2}\int_{\R^{3}} -\lambda_1^{-1}u \Delta u+\mu_1^{-1}\la D \phi,\phi\ra-|\phi|^2|u|^2 $$
and since $(U_{\lambda,\xi},\Psi_{\lambda,\xi,a})$ are critical points of $J_0$, then
$$ \tilde{W}_{\lambda,\xi,a}=\left(\mu_1^{-\frac{1}{2}}U_{\lambda,\xi},\lambda_1^{-\frac{1}{2}}\Psi_{\lambda,\xi,a}\right)$$
are critical points of $\tilde{J}_0.$
\end{remark}

\begin{lemma} \label{lem: sol f''}
We have $T_{W_0}Z=\text{ker}J''_0(W_0).$
\end{lemma}
\begin{proof}
It is standard to check that $T_{W_0}Z\subseteq\text{ker}J''_0(W_0)$, so it suffices to prove the inclusion $\text{ker} J''_0(W_0)\subseteq T_{W_0}Z.$ Moreover, since $\textnormal{dim}(T_{W_0}Z)=7$ it is enough to show that
$$\textnormal{dim}(\text{ker}J''_0(W_0))\leq 7$$
and, by means of Remark \ref{rmk: rescaling}, this is equivalent to
$$\textnormal{dim}(\text{ker}\tilde{J}''_0(\tilde{W}_0))\leq 7.$$
On the sphere $\Sph$, the linearization of \eqref{eq: pb on R rescaled} at $\tilde{W}_0$ reads as
\begin{equation}\label{eq: kerJ''}
\begin{cases}
L_{g_{\Sph}}\nu=\lambda_1\nu|\Psi_1|^2+2\lambda_1V_1\la\Psi_1,\eta\ra\\
\\
D_{g_{\Sph}}\eta=\mu_1|V_1|^2\eta+2\mu_1\nu V_1\Psi_1
\end{cases}
\end{equation}
where $(V_1,\Psi_1)=\left(\mu_1^{-\frac{1}{2}}(f^{-\frac{1}{2}}U_{\lambda,\xi})\circ \pi, \lambda_1^{-\frac{1}{2}}(f\circ\pi)^{-1}F(\Phi_{\lambda,\xi,a}\circ\pi)\right)=(1,\Psi_1)$. Notice that $\Psi_1$ satisfies
\begin{equation}\label{eq: psi1}
D_{g_{\Sph}}\Psi_1=\frac{3}{2}|\Psi_1|\Psi_1\quad \textnormal{and}\quad |\Psi_1|=1,
\end{equation}
so it is an eigenspinor of $D_{g_{\Sph}}$ with eigenvalue $\frac{3}{2}$. We set $\eta=\sum_{k\in \mathbb{Z}} f_{k}\Psi_{k}$ where $\Psi_{k}$ is a trivialization with Killing spinors and we write $f_{1}=g_{1}+ih_{1}$, where $g_1$ and $h_1$ are real valued functions. We will first find $f_{1}$. Since $f_1=\la\eta,\Psi_{1}\ra$, we have (see Lemma 5.2 and Formula 5.16 in \cite{I})
$$\Delta_{g_{\Sph}} f_{1}=\la\Delta_{g_{\Sph}} \eta, \Psi_{1}\ra+\la \eta,\Delta_{g_{\Sph}} \Psi_{1}\ra+\la D_{g_{\Sph}}\eta,\Psi_{1}\ra.$$
Notice now that, by \eqref{eq: psi1} and the Lichnerowicz's formula on the sphere
$$D^2_{g_{\Sph}}=-\Delta_{g_{\Sph}} +\frac{3}{2} , $$
we have $-\Delta_{g_{\Sph}} \Psi_{1}=\frac{3}{4}\Psi_{1}$ and
\begin{align*}
-\Delta_{g_{\Sph}} \eta&= D^{2}_{g_{\Sph}}\eta-\frac{3}{2}\eta\\
& =D_{g_{\Sph}}\left(\frac{3}{2}\eta +3\nu\Psi_{1}\right)-\frac{3}{2}\eta=\frac{3}{2}\left(\frac{3}{2}\eta+3\nu\Phi_{1}\right)+3\nabla \nu \cdot \Psi_{1}+\frac{9}{2}\nu\Psi_{1}-\frac{3}{2}\eta\\
&=\frac{3}{4}\eta+9\nu\Psi_{1}+3\nabla \nu\cdot \Psi_{1}.
\end{align*}
Therefore
\begin{equation}\label{eq: Delta f1}
\begin{split}
-\Delta_{g_{\Sph}} f_{1}&= \frac{3}{4}f_{1}+9\nu+3\la\nabla \nu\cdot \Psi_{1},\Psi_{1}\ra+\frac{3}{4}f_{1}-\frac{3}{2}f_{1}-3\nu\\
&=6\nu+3\la\nabla \nu\cdot \Psi_{1},\Psi_{1}\ra.
\end{split}
\end{equation}
Since the last addend in the previous equality is purely imaginary, we take the real and imaginary part to have
$$-\Delta_{g_{\Sph}} g_{1}=6\nu$$
and
$$-\Delta_{g_{\Sph}} h_{1}=-3i\la\nabla \nu\cdot \Psi_{1},\Psi_{1}\ra.$$
In particular, recalling that $L_{g_{\Sph}}=-\Delta_{g_{\Sph}}+\frac{3}{4}$ and the first equation in \eqref{eq: kerJ''}, we have the system
\begin{equation}
\begin{cases}
-\Delta_{g_{\Sph}} \nu=\frac{3}{2} g_{1}\\
\\
-\Delta_{g_{\Sph}} g_{1}=6\nu
\end{cases}
\end{equation}
Hence,
$$\Delta^{2}_{g_{\Sph}} g_{1}=9g_{1}$$
from which we deduce that $g_{1}$ is the first eigenfunction of the Laplacian on the sphere and $\nu=\frac{g_{1}}{2}$. So, the first equation in \eqref{eq: kerJ''} becomes
$$L_{g_{\Sph}}\frac{g_1}{2}=\frac{3}{4}\frac{g_1}{2}+3f_1 $$
and recalling the definition of $L_{g_{\Sph}}$, from the quality above we get
$$f_1=g_1.$$ Using this fact, the system \eqref{eq: kerJ''} becomes
\begin{equation}\label{new J''=0}
\begin{cases}
\nu=\frac{g_1}{2}\\
\\
D_{g_{\Sph}}\eta=\frac{3}{2}\eta+\frac{3}{2}\la \eta,\Psi_1\ra\Psi_1.
\end{cases}
\end{equation}
Hence we need to compute the dimension of
$$ \Lambda=\left\{\eta\in H^{\frac{1}{2}}(\Sigma\Sph)\; : \;D_{g_{\Sph}}\eta=\frac{3}{2}\eta+\frac{3}{2}\la \eta,\Psi_1\ra\Psi_1  \right\}.$$
This computation has been carried out by Isobe in \cite{I} for general dimensions of the sphere $\mathbb{S}^m$, so in our situation it suffices to take $m=3$ in \cite[Lemma 5.1]{I} to get $\textnormal{dim}(\Lambda)=7$ as desired.
\end{proof}

\noindent
Now we will focus on the reduced functional. For a fixed $a\in \Sigma\R^{3}$, with $|a|=1$, we set $V_{\lambda,\xi}=|U_{\lambda,\xi}|^2 |\Phi_{\lambda,\xi,a}|^{2}$, so that $V_{\lambda,\xi}(x)=\frac{1}{\lambda^{3}}V_{1,0}(\frac{1}{\lambda}(x-\xi))$ and let
$$\Gamma(\lambda,\xi)=\frac{1}{2}\int_{\R^{3}}h(x)V_{\lambda,\xi}(x)dx,$$
for $(\lambda,\xi)\in (0,+\infty)\times \R^{3}$. Then we have the following

\begin{proposition} \label{pro: Gammaestimates}
$\Gamma$ is of class $C^2$ on $(0,+\infty)\times \R^{3}$ and it can be extended to a $C^{1}$ function at $\lambda=0$ by
$$\Gamma(0,\xi)=c_{0}h(\xi), \qquad c_0=\frac{1}{2}\int_{\R^{3}}V_{1,0}(x)dx$$
Also,
$$\lim_{\lambda\to 0}\nabla^{2}_{\xi}\Gamma(\lambda,\xi)=c_{0}\nabla^{2}h(\xi),$$
uniformly on every compact of $\R^{3}$.
Moreover, for any compact set $\Sigma\subset \R^{3}$, there exists a constant $C=C_{\Sigma}$ such that
$$|\partial_{\lambda}\Gamma(\lambda,\xi)-c_{1}\lambda \Delta h(\xi)|\leq C_{\Sigma}\lambda^{2},$$
for all $\lambda>0$ and all $\xi\in \Sigma$, being
$$c_{1}=\int_{\R^{3}}|y|^{2}V_{1,0}(y)dy \; .$$
\end{proposition}

\begin{proof}
We have by a change of variable that
$$\Gamma(\lambda,\xi)=\int_{\R^{3}}h(\lambda x+\xi)V_{1,0}(x)dx$$
Using the smoothness of $h$ and the dominated convergence, we have that
$$\lim_{\lambda\to 0}\Gamma(\lambda,\xi)=c_{0}h(\xi).$$
The same reasoning applies to show that one has
$$\nabla_{\xi} \Gamma(0,\xi)=c_{0}\nabla h(\xi);\qquad
\nabla_{\xi}^{2}\Gamma(0,\xi)=c_{0}\nabla^{2}h(\xi) \quad \text{ and } \quad \nabla_{\lambda}\Gamma (0,\xi)=0.$$
The last equality follows from the oddness of the integral, that is
$$\int_{\R^{3}}x_{i}V_{1,0}(x)dx=0 \; , \quad i=1,2,3.$$
We fix now a compact set $\Sigma$, then by Taylor expansion of $y\mapsto h(y+\xi)$, we have
$$\left|\partial_{\xi_{i}}h(y+\xi)-\partial_{\xi_{i}}h(\xi)-\sum_{j=1}^{3}\partial^{2}_{\xi_{i}\xi_{j}}h(\xi)y_{j}\right|\leq C_{\Sigma}|y|^{2} \; .$$
Also, notice that since
$$\int_{\R^{3}}y_{i}y_{j}V_{1,0}(y)dy=0 \; , \text{ if } i\not=j,$$
we have for our choice of $c_1$:
$$c_{1}\lambda \Delta h(\xi)=\int_{\R^{3}}\sum_{i=1}^{3} \left(\partial_{\xi_{i}}h(\xi)+\sum_{j=1}^{3}\partial^{2}_{\xi_{i}\xi_{j}}h(\xi)\lambda y_{j}\right)y_{i} V_{1,0}(y)dy.$$
Therefore
$$|\partial_{\lambda}\Gamma(\lambda,\xi)-c_{1}\lambda \Delta h(\xi)|\leq C_{\Sigma}\lambda^{2}.$$
\end{proof}

\begin{proposition}
Let $k$ and $h$ be functions as in the main Theorem \ref{thm: main}. Then there exists an open set $\Omega \subset (0,+\infty) \times \R^{3}$ such that $\nabla \Gamma \not=0 $ on $\partial \Omega$ and
$$deg(\nabla \Gamma, \Omega, 0)=\sum_{\substack{\xi \in \textnormal{crit}[h]\\ \Delta h(\xi)<0}} (-1)^{m(h,\xi)}+1.$$
\end{proposition}

\begin{proof}
Let $s>0$, we consider the set $$\mathcal{B}_{s}=\left\{(\lambda,\xi)\in (0,+\infty)\times \R^{3}; |(\lambda,\xi)-(s,0)|\leq s-\frac{1}{s}\right\}.$$
We will show that for $s$ large enough, we can choose $\Omega=\mathcal{B}_{s}$. First, we set
$$\textnormal{crit}[h]= \{\xi^{1}, \xi^{2},\cdots, \xi^{l}\},$$
for some $l \in \mathbb{N}$.  Since the south pole is not a critical point for $k$, we have that for $r$ large enough
$$\textnormal{crit}[h]\subset A_{r}=\left\{\xi\in \R^{3};|\xi|\leq r\right\}.$$
Since $h$ is a Morse function (as well as $k$), then by the non-degeneracy condition $(i)$, there exist constants $\mu\in(0,r)$ and $\delta>0$ such that
$$ \left|\Delta h(\xi)\right|>\delta, \quad \forall \xi \in \bigcup_{i=1}^{l}B_{\mu}\left(\xi^{i}\right),$$
where $B_{\mu}\left(\xi^{i}\right)$ denote as usual the balls of centers $\xi^i$ and radius $\mu$. By using Proposition \ref{pro: Gammaestimates}, we have that for $s$ sufficiently large and $\mu$ even smaller if necessary,
$$\partial_{\lambda} \Gamma(\lambda, \xi) \not =0, \text{ in } \partial \mathcal{B}_{s}\cap \left((0,\mu)\times \bigcup_{i=1}^{l}B_{\mu}\left(\xi^{i}\right)\right).$$
Hence,
$$\nabla \Gamma \not=0 \text{ in } \partial \mathcal{B}_{s}\cap \left((0,\mu)\times \bigcup_{i=1}^{l}B_{\mu}\left(\xi^{i}\right)\right).$$
Again, by Proposition \ref{pro: Gammaestimates}, since $\Gamma$ extends to a $C^{1}$ function at $\lambda=0$ and $\nabla_{\xi}\Gamma(0,\xi)=c_{0}\nabla h(\xi)$, we have that
$$\nabla \Gamma \not=0 \text{ in } \partial \mathcal{B}_{s} \cap \left((0,\mu)\times A_{2r}\setminus \bigcup_{i=1}^{l} B_{\mu}\left(\xi^{i}\right)\right).$$
Hence,
$$\nabla \Gamma \not=0 \text{ in } \partial \mathcal{B}_{s} \cap ((0,\mu)\times A_{2r}).$$
So it remains to study $\Gamma$ on the component of $\partial \mathcal{B}_{s}$ outside $(0,\mu)\times A_{2r}$. So we consider the Kelvin reflection
$$\tau: \R^{3}\setminus \{0\}\to \R^{3}\setminus \{0\}, \qquad \tau(x)=\frac{x}{|x|^{2}}.$$
We notice that
$$\tau^{*}(g_{\R^{3}})=\frac{1}{|x|^{4}}g_{\R^{3}}.$$
Hence, for all $F\in L^{6}(\R^{3})$, by putting $y=\tau(x)$, we have
$$\int_{\R^{3}}h(y)|F(y)|^{6}dy=\int_{\R^{3}}h(\tau(x))|F(\tau(x))|^{6}f^{3}(x)dx=\int_{\R^3}h(\tau(x))|F^{*}(x)|^{6}dx,$$
where
$$F^{*}(x)=\frac{1}{|x|^{2}}F\left(\frac{x}{|x|^{2}}\right).$$
In particular, if we set
$$\tilde{\lambda}=\frac{\lambda}{\lambda^{2}+|\xi|^{2}}, \qquad \tilde{\xi}=\frac{\xi}{\lambda^{2}+|\xi|^{2}}, $$
we have that
$$|V_{\lambda,\xi}^{*}(x)|=|V_{\tilde{\lambda},\tilde{\xi}}(x)|.$$
So we define $$\tilde{\Gamma}=\frac{1}{2}\int_{\R^{3}}h(\tau(x))|V_{\lambda,\xi}(x)|dx, $$
and we have that
$$\Gamma(\lambda, \xi)=\tilde{\Gamma}(\tilde{\lambda},\tilde{\xi}).$$
Once again, by using Proposition \ref{pro: Gammaestimates}, we have that $\tilde{\Gamma}$ can be extended to a $C^{1}$ function up to the origin $(0,0)\in [0,\infty)\times \R^{3}$. Since $(\lambda,\xi)\mapsto (\tilde{\lambda},\tilde{\xi})$ is a diffeomorphism, then $\nabla \Gamma(\lambda,\xi)=0$ if and only if $\tilde{\Gamma}(\tilde{\lambda},\tilde{\xi})=0$. But by assumption, the south pole is not a critical point of $h$, hence $0$ is not a critical point of $h(\tau(x))$. Therefore, $\nabla \tilde{\Gamma}\not=0$ in a neighborhood of the origin and so $\nabla \Gamma\not=0$ in a neighborhood of infinity. Finally, we have that for $r$ and $s$ large enough,
$$\nabla \Gamma\not=0 \text{ on }\partial \mathcal{B}_{s}\setminus \left((0,\mu)\times \overline{A}_{2r}\right).$$
The degree computation is by now standard and it follows for instance as in \cite{gottlieb}.
\end{proof}

\begin{remark}\label{rmk: degeneracy}
We want explicitly to notice that at this point we cannot directly conclude as in the classical cases (see for instance \cite{ambgarper, malchiodi-uguzzoni 2002}), since the critical points of $\Gamma$ on $Z$ are degenerate: this is due to the invariance of the functional with respect to the parameters $a$ and this degeneracy causes the degree to vanish.
\end{remark}

\noindent
We recall that $Z$ is a non-degenerate manifold of critical points of $J_{0}$ and $J''_{0}$ is Fredholm of index zero, therefore we have that there exists $\varepsilon>0$ such that for all $z\in Z_{c}\subset Z$ with $Z_{c}$ compact, there exists a unique $w(z)\in T_{z}Z^{\perp}$ such that
$$PJ'_{\varepsilon}(z+w(z))=0$$
where $P:A\to T_{z}Z^{\perp}$ is the orthogonal projection. Now, to find a solution to our problem, it is enough to find a critical point for the function $\Phi_{\varepsilon}:Z\to \R$ defined by
$$\Phi_{\varepsilon}(z)=J_{\varepsilon}(z+w(z)).$$
In order to do this, we will consider the set of the parameters $a$
$$\left\{ a\in \Sigma\R^{3} \; : \quad |a|=1 \right\}\simeq \Sph $$
as a Lie group. Hence, we will consider the natural action of $\Sph$ on $Z \simeq (0,+\infty)\times \R^{3}\times \Sph$, being $Z$ parameterized by $(\lambda, \xi, a)$. Also, we notice that $(J_{0})_{|Z}$ and $G_{|Z}$ are invariant under this action: then we need to extend the action to the whole space $D^{\frac{1}{2}}(\Sigma\R^{3})$. In order to do this, we recall that the spinor bundle of $\R^{3}$ can be trivialized by Killing spinors that are either constant (parallel spinors) or spinors of the form $x\cdot \phi$ with $\phi$ constant. So we fix an orthonormal basis of $\Sigma \R^{3}$ of the form $$\left\{a_{1},a_{2},x\cdot a_{1},x\cdot a_{2}\right\},$$
where $a_1,a_2$ are (distinct) constant spinors with $|a_1|=|a_2|=1$. Hence, if $\phi\in D^{\frac{1}{2}}(\Sigma\R^{3})$, there exist $f_{1},f_{2},g_{1},g_{2}$ such that
$$\phi(x)=(f_{1}(x)+g_{1}(x)x)\cdot a_{1}+(f_{2}(x)+g_{2}(x)x)\cdot a_{2}.$$
Since $a_{1}$ and $a_{2}$ can be seen as elements in $\Sph$, we can define the action for a general $w\in \Sph$ and $\phi \in D^{\frac{1}{2}}(\Sigma\R^{3})$ by
$$w\phi=(f_{1}(x)+g_{1}(x)x)\cdot w a_{1}+(f_{2}(x)+g_{2}(x)x)\cdot wa_{2}.$$
In this way, this last action extends the one previously defined on $Z$ and in addition both $J_{0}$ and $G$ are invariant under this action. Therefore, $\Phi_{\varepsilon}$ descends to a $C^{1}$ function $\tilde{\Phi}_{\varepsilon}$ defined on the quotient
$$Z/{\Sph}\simeq (0,\infty)\times \R^{3}.$$
The same argument works for $\Gamma$; therefore for $\varepsilon$ small enough, we have that
$$\tilde{\Phi}'_{\varepsilon}=\varepsilon \Gamma +o(\varepsilon).$$
At this point, from the invariance of the degree by homotopy, we have that
$$\deg(\tilde{\Phi}'_{\varepsilon},\mathcal{B}_{s},0)=\sum_{\substack{\xi \in \textnormal{crit}[h]\\ \Delta h(\xi)<0}} (-1)^{m(h,\xi)}+1.$$
Finally, by assumption $(ii)$, by contradiction (for the argument, see for instance \cite{AB, ambgarper}) if
$$\sum_{\substack{\xi \in \textnormal{crit}[h]\\ \Delta h(\xi)<0}} (-1)^{m(h,\xi)} \neq -1,$$
then $\tilde{\Phi}_{\varepsilon}$ has a critical point that can be lifted as a critical orbit of $\Phi_{\varepsilon}$, which in turn ends the proof of the main Theorem \ref{thm: main}.

\end{document}